\documentclass[11pt]{amsart}
\usepackage{amsmath,amssymb,mathrsfs,color,amsthm}
\usepackage[title]{appendix}
\usepackage{amssymb}
\usepackage{array}
\usepackage{longtable}
\usepackage{tabularx}
\usepackage{booktabs}
\usepackage{graphicx}  
\usepackage{graphicx}  
\usepackage{caption}   
\usepackage{enumitem}
\usepackage{relsize}
\usepackage{graphicx}
\usepackage{xcolor}
\usepackage{listings}
\usepackage[labelfont=bf,labelsep=period]{caption}
\allowdisplaybreaks

\def\R{{\mathbb R}}
\def\P{{\mathbb P}}
\def\Q{{\mathbb Q}}
\def\E{{\mathbb E}}

\def\var{\varepsilon}

\def\maL{\mathcal{L}}
\newtheorem{thm}{Theorem}[section]
\newtheorem{cor}[thm]{Corollary}
\newtheorem{lem}[thm]{Lemma}
\newtheorem{prop}[thm]{Proposition}

\theoremstyle{definition}
\newtheorem{de}[thm]{Definition}
\theoremstyle{remark}
\newtheorem{rem}[thm]{Remark}

\newtheorem{ass}[thm]{\bfseries Assumption}

\numberwithin{equation}{section}

\newcommand{\rmd}{{\rm d}}
\newcommand{\rme}{{\rm e}}
\newcommand{\rmdiv}{{\rm div}}
\allowdisplaybreaks

\begin{document}

\title[Convergence rates of Wasserstein gradient flows]{Convergence rates of Wasserstein gradient flows for nonlinear Fokker-Planck equations with mobility and related inequalities}

\author{Zhenxin Liu}
\address{Z. Liu: School of Mathematical Sciences, Dalian University of Technology, Dalian
116024, P. R. China}
\email{zxliu@dlut.edu.cn}

\author{Xuewei Wang}
\address{X. Wang (Corresponding author): School of Mathematical Sciences, Dalian University of Technology, Dalian
116024, P. R. China}
\email{Wangxueweii@163.com}

\date{July 11, 2026}

\subjclass[2020]{35Q84, 60H10, 49Q22}

\keywords{nonlinear Fokker-Planck equation, McKean-Vlasov SDE, Wasserstein gradient flow, rate of convergence, HWI inequality.}

\begin{abstract}
For nonlinear Fokker-Planck equations with mobility, the Wasserstein gradient flow structure is described by the generalized relative entropy as the energy functional and the modified Wasserstein metric $W_h$ as the associated metric structure.
This work investigates the nonlinear effects induced by mobility and establishes the corresponding inequalities.
For nonlinear diffusion, we establish a logarithmic Sobolev inequality, which yields the convergence rate of the free energy functional and the Talagrand inequality.
By further exploiting the relationship between the weighted homogeneous Sobolev norm and the $W_h$ metric, we derive an HWI inequality relating the relative entropy, the $W_h$ metric, and the Fisher information.
In the case of mobility dependent drift and linear diffusion, the convergence rate in the $W_h$ metric is also obtained by applying the Girsanov theorem.
\end{abstract}

\maketitle

\section{Introduction}
\setcounter{equation}{0}
 In this paper, we consider the nonlinear Fokker-Planck equation
\begin{equation}\label{NFPE}
\begin{aligned}
\partial_t p_t &= {\rm div} \big(\nabla \Phi h(p_t) \big) + \Delta f(p_t),~ t\in [0, +\infty),~ x\in \R^d
\end{aligned}
\end{equation}
with initial condition $p_0$, where $h(r):=rb(r)$.
The density corresponding to the solution to the McKean-Vlasov stochastic differential equation (SDE)
\begin{equation}\label{MVSDE}
\begin{aligned}
\rmd X_t&=-\nabla \Phi(X_t) b\big(p_t( X_t)\big) \rmd t + \sqrt{\frac{2f\big(p_t( X_t)\big)}{p_t( X_t)}}\rmd W_t, ~ t \geq 0, \\
X_0&=\varsigma_0
\end{aligned}
\end{equation}
on $\R^d$ satisfies \eqref{NFPE}.
Here, $W_t$, $t\geq 0$, is a standard Brownian motion on a probability space $(\Omega, \mathcal{G}, \mathbb{P})$ with normal filtration $(\mathcal{G}_t)_{t\geq 0}$ and $\varsigma_0$ has density $p_0$.
The notation $p_t(X_t)$ denotes the value at $X_t$ of the density function $p_t$ associated with the distribution $P_t:={\rm Law}_{\P}(X_t)$. In other words, $p_t(X_t):=\frac{\rmd P_t}{\rmd x}(X_t)$, $t \geq 0$.
Under the appropriate conditions on the functions $\Phi: \R^d \rightarrow [0, +\infty)$ and $b, f:\R \rightarrow [0, +\infty)$, which respectively characterize the potential and the influence of the distribution on the drift and diffusion terms, \eqref{NFPE} admits a unique smooth solution, and \eqref{MVSDE} admits a pathwise unique strong solution (as discussed, for instance, in \cite{BR, G}). The subsequent analysis is carried out within this framework.

As stated in \cite{CHR}, mobility is often used to model the prevention of overcrowding.
It also gives rise to a particular gradient structure
\begin{equation}\label{GF}
\partial_t p_t=-{\rm  grad}_{W_h} \mathcal{F}(p_t)={\rm div}\Big( h\big(p_t \big) \nabla \big( \dfrac{\delta \mathcal{F}(p_t)}{\delta p} \big) \Big)
\end{equation}
for equation \eqref{NFPE}.
Here, the free energy functional
\begin{equation}\label{energy}
\mathcal{F}(p_t) :=\int_{\R^d} \eta (p_t) +\Phi p_t \rmd x
\end{equation}
consists of the generalized entropy $\int_{\R^d} \eta (p_t) \rmd x$ and the potential energy $\int_{\R^d} \Phi p_t \rmd x$, and the appropriate metric is the modified Wasserstein metric
\begin{equation}\label{Whdef}
\begin{aligned}
W_h^2(\mu_0, \mu_1):=\inf \Big\{& \int_0^1 \!\! \int_{\R^d} |v_t|^2 \rmd \hat{h}(\mu_t) \rmd t \Big| \partial_t \mu_t+{\rm div} \big(v_t \hat{h}(\mu_t) \big) =0  \\
& \text{ holds in the distributional sence} \Big\},
\end{aligned}
\end{equation}
where $\eta(r):=\int_0^r g(s)\rmd s := \int_0^r \int_1^s \frac{f'(\tau)}{h(\tau)} \rmd \tau \rmd s $ and
$$\mu_0, \mu_1 \in \mathcal{P}_{ac}(\R^d):= \Big\{P \text{ is a probability measure on } \R^d \text{ and } P\ll \rmd x \Big \}.$$
The map $\hat{h}$, induced by $h$, is defined on pairs of measures. Its precise definition is given in Subsection \ref{nota}.
In this setting, the classical results for the linear and nonlinear cases (see, for instance, \cite{AGS, CMV}) are not directly applicable to the convergence analysis.
The generalized relative entropy satisfies
\begin{equation}\label{Hg}
H_g(P_t|P_{\infty}) :=\int_{\R^d} \eta(p_t) -\eta(p_{\infty}) +\Phi (p_t- p_{\infty}) \rmd x=\mathcal{F}(p_t)-\mathcal{F}(p_{\infty}),
\end{equation}
where
\begin{equation}\label{pinfty}
\rmd P_{\infty}=p_{\infty} \rmd x :=g^{-1}(-\Phi+ c_0) \rmd x
\end{equation}
and the constant $c_0$ be given in Assumption \ref{reg}.
Based on this relation, we establish the convergence of the modified Wasserstein gradient flow and the related inequalities by analyzing the relationship among $H_g(P_t|P_{\infty})$, $W_h(P_t, P_{\infty})$, and the relative modified Fisher information
\begin{equation}\label{Ig}
I_g(P_t|P_{\infty})
:= \int_{\R^d} |\nabla g(p_t)+ \nabla \Phi|^2 h(p_t) \rmd x.
\end{equation}
The details of the modified Wasserstein gradient flow can be found in \cite{LW}.

However, it is highly challenging to establish these inequalities on the whole space.
On the one hand, our study focuses on the convergence of functionals that fail to be displacement convex.
Reference \cite{CLSS} examines the properties of the potential energy $\int_{\R^d} \Phi p \rmd x$ and presents an example demonstrating the failure of $W_h$-displacement convexity. As a result, methods relying on displacement convexity are not applicable to the convergence analysis, and geometric approaches are difficult to apply directly in this setting.
On the other hand, the direct calculation methods for \eqref{NFPE} or \eqref{MVSDE} used in \cite{AMTU} and \cite{FJ} are difficult to implement in the present setting. The main difficulty arises from the nonlinear structure, which prevents effective control of the integrability of $|\nabla p_t|$. Due to the presence of the mobility, terms of higher order in $|\nabla p_t|$ arise. This makes termwise verification of integrability highly difficult.
Fortunately, the square integrability of the logarithmic gradient $\frac{|\nabla p_t|}{p_t}$ with respect to $P_t$ can be guaranteed. This provides a meaningful foundation for establishing the inequalities and convergence results.

We do not use either of the two classical approaches discussed above. Instead, we restrict our analysis to a suitable admissible set. The key question is how to choose this set. It must provide effective control over the density function while ensuring that the solution remains within the set. This idea is also present in \cite{EGO} and \cite{MR}. Motivated by these considerations, together with the roles of $f$ and $b$, we restrict the initial value to lie in the set
\begin{equation}\label{A}
  \Lambda_{\var}^g:=\{ P=p\maL^d \big| \|g(p)-g(p_{\infty}) \|_{L^{\infty}} \leq \var \},
\end{equation}
and prove that $P_t\in \Lambda_{\var}^g$, where $\maL^d$ denotes the Lebesgue measure and $\var>0$ is a constant.
It also guarantees that
$$m_{\var}p_{\infty} \leq p_t \leq  M_{\var} p_{\infty}$$
for some positive constants $m_{\var}$ and $M_{\var}$.
This admissible set cannot be replaced by a ball in the $W_h$ metric, as demonstrated by the counterexample in Remark \ref{AWh}.
If $p_{\infty}$ has sufficient regularity, the desired inequalities can be established by means of a direct estimate, which then yields convergence. However, this estimate is rather rough. To obtain a sharper result, we retain the mobility $h$ throughout the estimate.

This paper establishes three main results.
The first main result is a logarithmic Sobolev inequality of the form
\begin{equation*}
H_g(P_t|P_{\infty}) \leq C_{HI} I_g(P_t|P_{\infty}),
\end{equation*}
where the constant $C_{HI}$ can be found in Theorem \ref{HI}.
Moreover, using the dissipation equation and the metric derivative for the $W_h$ metric in \cite{LW}, we obtain the convergence rate of the relative entropy
\begin{equation*}
H_g(P_t|P_{\infty}) \leq \rme^{-\frac{t}{C_{HI}}}H_g(P_0|P_{\infty}),
\end{equation*}
and the Talagrand inequality
\begin{equation*}
  W_h^2(P_t, P_{\infty}) \leq C_T H_g(P_t|P_{\infty}),
\end{equation*}
where $C_T=4C_{HI}$.
This further implies exponential convergence of the free energy functional $\mathcal{F}$ along the gradient flow \eqref{GF}, namely
\begin{equation*}
\mathcal{F}(p_t)-\mathcal{F}(p_{\infty})\leq \rme^{-\frac{t}{C_{HI}}} (\mathcal{F}(p_0)-\mathcal{F}(p_{\infty})).
\end{equation*}
The second main result is the HWI inequality. Based on the weighted homogeneous Sobolev norm
\begin{equation*}
  \|\nu\|_{\dot{H}^{-1}(\pi)}:=\sup \Big\{ |\langle \phi, \nu \rangle | \Big| \int_{\R^d} |\nabla \phi|^2 \rmd \pi \leq 1 \Big\},
\end{equation*}
and the semi-norm
\begin{equation*}
  D_{\pi}(P, Q):=\|P-Q\|_{\dot{H}^{-1}(\pi)},
\end{equation*}
we prove the HWI-like inequality
\begin{equation*}
    H_g(P_t|P_{\infty}) \leq C_{HWI}' D_{\hat{h}(P_{\infty})}(P_t, P_{\infty}) \sqrt{I_g(P_t|P_{\infty})},
  \end{equation*}
where $\pi$ is a given reference measure and
the constant $C_{HWI}'$ is given in Theorem \ref{HWI1}.
By further establishing a connection among the $\dot{H}^{-1}$ norm, $W_2$, and $W_h$, we derive the HWI inequality
\begin{equation*}
  H_g(P_t|P_{\infty}) \leq C_{HWI} W_h(P_t, P_{\infty}) \sqrt{I_g(P_t|P_{\infty})}
\end{equation*}
with the constant $C_{HWI}$ defined in Theorem \ref{HWI}.
The final main result is the convergence rate in the $W_h$ metric for the gradient flow \eqref{GF}. Since the proof relies on the Girsanov theorem, we restrict our attention to the case of mobility dependent drift and linear diffusion.
Under these conditions, the log-Harnack inequality can be established. Together with the convergence rate of the relative entropy and the Talagrand inequality,
this yields exponential convergence in the $W_h$ metric, namely
\begin{equation*}
W_h^2(P_t|P_{\infty}) \leq \rme^{-\frac{t-\tau}{C_{HI}}} C_T C_{HW} W_h^2(P_0, P_{\infty}),
\end{equation*}
where the constant $C_{HW}$ is given in Theorem \ref{WHcon}.

This work is organized as follows. Section \ref{pre} introduces the preliminaries.
Section \ref{mainnon} considers the case involving nonlinear diffusion and mobility.
In Subsection \ref{LSI}, we study the logarithmic Sobolev inequality and the convergence rate of the free energy along gradient flow \eqref{GF}.
The HWI inequality is given in Subsection \ref{mainHWI}.
The case of mobility dependent drift and linear diffusion is studied in Section \ref{lin}, where the convergence rate in the $W_h$ metric for gradient flow \eqref{GF} is established.

\section{Preliminaries}\label{pre}

\subsection{Notation and assumptions}\label{nota}
In this section, we introduce the notation and assumptions used throughout the paper, together with some results from \cite{LW}.

\begin{longtable}{@{}p{0.13\textwidth} p{0.59\textwidth} p{0.18\textwidth}@{}}
\toprule
\textbf{Symbol} & \textbf{Meaning} & \textbf{Reference} \\
\midrule
\endfirsthead

\toprule
\textbf{Symbol} & \textbf{Meaning} & \textbf{Reference} \\
\midrule
\endhead

$T$ & terminal time in $(0, +\infty)$ &  \\
$\maL^d$ & Lebesgue measure on $\R^d$ & \\
${\rm I_{d\times d}}$ & $d\times d$ identity matrix & \\
$\mathcal{M}_{ac}^+$ & space of finite, nonnegative Borel measures & \\
$\mathcal{P}_{ac}$ & space of absolutely continuous probability measures & \\
$\mathcal{P}_2$ & space of probability measures with finite second moment & \\
$\mathcal{P}_{ac, 2}$ & intersection of sets $\mathcal{P}_2 $ and $\mathcal{P}_{ac}$ & \\
$\mathcal{B}_b^+$ & space of positive, bounded Borel measurable functions & \\
$\mathcal{F}$ & free energy functional & \eqref{energy} \\
$W_h$ & modified Wasserstein metric & \eqref{Whdef} \\
$H$ & classical relative entropy & \eqref{classH} \\
$H_g$ & generalized relative entropy & \eqref{Hg} \\
$I_g$ & relative modified Fisher information & \eqref{Ig} \\
$P_{\infty}, p_{\infty}$ & invariant measure and its density & \eqref{pinfty} \\
${\rm Tan}_{\hat{h}, P_t}\mathcal{P}_{ac, 2}$ & tangent space of $\mathcal{P}_{ac, 2}$ with respect to $W_h$ at $P_t$ & Assumption \ref{met} \\
$\Lambda_{\var}^g$ & admissible subset & \eqref{A} \\
$\| \cdot \|_{\dot{H}^{-1}}(\pi)$ & weighted homogeneous Sobolev norm with respect to $\pi$ & Definition \ref{sobolev} \\
$D_{\pi}$ & semi-norm induced by $\dot{H}^{-1}(\pi)$ & \eqref{semi} \\
\bottomrule
\end{longtable}

For the smooth function $h: [0, +\infty) \rightarrow [0, +\infty)$, we define
\begin{equation*}
\begin{aligned}
\hat{h}: \mathcal{P}_{ac}(\R^d) \rightarrow \mathcal{M}_{ac}^+(\R^d), ~\mu \mapsto \hat{h}(\mu),
\end{aligned}
\end{equation*}
where $\rmd \mu=u \rmd x$ and $\rmd \hat{h}(\mu):= h(u)\rmd x$.

\begin{ass}[Regularity conditions]\label{reg} \quad
\begin{enumerate}[label=(\roman*)]
  \item \label{f} The smooth function $f:\R\rightarrow [0, +\infty)$ satisfies $f(0)=0$ and $\gamma_1 \leq f'(r) \leq \gamma_2 $ for all $r \in \R$, where $0 < \gamma_1 <\gamma_2 <+\infty$.
  \item \label{b} The smooth function $b:\R\rightarrow [0, +\infty)$ satisfies $b_1 \leq b(s) \leq b_2$ for all $s\in \R$, where $0 < b_1 <b_2<+\infty$.
  \item \label{phi} The smooth function $\Phi:\R^d\rightarrow [0, +\infty)$ satisfies that there exist constants $C, R \geq 0$ such that $|\nabla \Phi(x)|\leq C|x| $ for all $|x|>R$.
  \item \label{initial} The distribution $P(0)\in \mathcal{P}_{ac, 2}(\R^d)$ has probability density function $p_0 \in L^{\infty}(\R^d)$ such that $\mathcal{F}(p_0) <M_0$ for some constant $M_0\in \R$.
  \item \label{smooth} The nonlinear Fokker-Planck equation \eqref{NFPE} has a unique smooth solution satisfying
  $$\quad \quad p \in L^{\infty}([0,+\infty) \times \R^d), ~p_t(x)\geq 0, ~\int_{\R^d}p_t(x)\rmd x =\int_{\R^d}p_0(x)\rmd x, ~  t\in [0, +\infty),~ x\in \R^d,$$
      and admits a unique stationary solution of the form
      \begin{equation*}
      p_{\infty} =g^{-1}(-\Phi+ c_0) ~ \text{  for some } c_0 \in \R.
      \end{equation*}
  \item \label{strong} The McKean-Vlasov SDE \eqref{MVSDE} has a pathwise unique strong solution $(X_t)_{0\leq t\leq T}$.
  \item \label{poin} The measure $\mu$ satisfies Poincar\'e inequality with the Poincar\'e constant $C_P$, where $\rmd \mu:=h(p_{\infty})\rmd x$.
\end{enumerate}
\end{ass}

\begin{rem}
The validity of conditions (i)--(vi) in Assumption \ref{reg}, together with a specific example, has already been presented in \cite{LW}. Therefore, no further elaboration is necessary here. We next focus on condition (vii). Let $V:=-\ln h(p_{\infty})$, so that $\rmd \mu=\rme^{-V}\rmd x$.
By
$$\nabla p_{\infty} =-\frac{\nabla \Phi}{g'(p_{\infty})}=-\frac{h(p_{\infty}) \nabla \Phi}{f'(p_{\infty})},$$
we obtain
\begin{equation*}
\nabla V=\dfrac{h'(p_{\infty}) \nabla \Phi}{f'(p_{\infty})}
\end{equation*}
and
\begin{equation*}
  \nabla^2 V=\dfrac{h'(p_{\infty})}{f'(p_{\infty})}\nabla^2 \Phi -\dfrac{h(p_{\infty})}{f'(p_{\infty})} \Big( \dfrac{h'(p_{\infty})}{f'(p_{\infty})} \Big)' \nabla \Phi \otimes\nabla \Phi.
\end{equation*}
Proposition 2.1 in \cite{BL} and Corollary 1.6 in \cite{BBCG} provide two sufficient conditions for (vii), namely
$$\nabla^2 \Phi \geq \kappa I,~ \dfrac{h'(p_{\infty})}{f'(p_{\infty})}\geq \kappa,~ \Big( \dfrac{h'(p_{\infty})}{f'(p_{\infty})} \Big)' \leq 0 \text{ for constant } \kappa >0,$$
or
$$\dfrac{h'(p_{\infty})}{f'(p_{\infty})}\langle x, \nabla \Phi \rangle \geq \kappa |x| \text{ for all } |x|\geq R, \text{ constant } \kappa >0. $$
A direct calculation shows that $$b(r)=\frac{2+r}{1+r},~ h(r)=\frac{2r+r^2}{1+r}$$
is an example satisfying these requirements.
Equation (1.7) of \cite{BBC} provides another such example.
\end{rem}

Following the definitions of generalized entropy and modified Fisher information (see, for instance, \cite{N}),
the generalized relative entropy of $P_t$ with respect to $P_{\infty}$ is defined by
\begin{equation*}
H_g(P_t|P_{\infty}) =\int_{\R^d} \eta(p_t) -\eta(p_{\infty}) +\Phi (p_t- p_{\infty}) \rmd x,
\end{equation*}
and the relative modified Fisher information of $P_t$ with respect to $P_{\infty}$ is given by
\begin{equation*}
I_g(P_t|P_{\infty})
= \int_{\R^d} |\nabla g(p_t)+ \nabla \Phi|^2 h(p_t) \rmd x.
\end{equation*}
The modified Wasserstein metric between $\mu_0, \mu_1 \in \mathcal{P}_{ac}(\R^d)$ is defined as
\begin{equation*}
\begin{aligned}
W_h^2(\mu_0, \mu_1)=\inf \Big\{& \int_0^1 \!\! \int_{\R^d} |v_t|^2 \rmd \hat{h}(\mu_t) \rmd t \Big| \partial_t \mu_t+{\rm div} \big(v_t \hat{h}(\mu_t) \big) =0  \\
& \text{ holds in the distributional sence} \Big\}.
\end{aligned}
\end{equation*}
Further details of the modified Wasserstein metric can be found in \cite{CLSS, DNS}.
Since this notion will be used later and is useful for comparison with $H_g(P_t|P_{\infty})$, we also present the definition of the classical relative entropy of $P_t$ with respect to $P_{\infty}$ by
\begin{equation}\label{classH}
H(P_t|P_{\infty}) :=\int_{\R^d} \ln \Big(\dfrac{p_t}{p_{\infty}}\Big) \rmd P_t.
\end{equation}

Proposition 2.3 and Lemma 2.4 of \cite{LW} imply that, under Assumption \ref{reg}, we have
\begin{equation}
P_t\in \mathcal{P}_2(\R^d), \quad t\in [0, T]
\end{equation}
and
\begin{equation}
\int_0^T \!\! \int_{\R^d} \dfrac{|\nabla p_t(x)|^2}{p_t(x)} \rmd x \rmd t<+\infty.
\end{equation}
This indicates that the generalized relative entropy $H_g(P_t| P_{\infty})$ and the relative modified Fisher information $I_g(P_t| P_{\infty})$ are finite. Therefore, in the subsequent analysis, we no longer discuss the cases in which these quantities take the value $+\infty$. Without loss of generality, we also assume that $W_h(P_0, P_{\infty})<+\infty$.
Furthermore, by Corollary 3.9 of \cite{LW}, the relative entropy dissipation equation holds in the form
\begin{equation}\label{rela}
\dfrac{\rmd }{\rmd t}H_g(P_t|P_{\infty})=-I_g(P_t|P_{\infty}).
\end{equation}

To ensure the validity of the $W_h$ metric and its metric derivative, we introduce the following assumptions.

\begin{ass}[Additional conditions for the validity of $W_h$]\label{met} \quad
\begin{enumerate}[label=(\roman*)]
  \item The smooth function $h$ is increasing and concave.
  \item The field
         $$v_t\in {\rm Tan}_{\hat{h}, P_t}\mathcal{P}_{ac, 2}(\R^d):=\overline{\{\nabla \zeta : \zeta \in C_c^{\infty}(\R^d) \}}^{L^2(\R^d, \R^d ; \hat{h}(P_{t}))},$$
         where $v_t:=-\nabla \big ( g(p_t) +\Phi \big)$.
\end{enumerate}
\end{ass}

Under Assumptions \ref{reg} and \ref{met}, the Wasserstein metric derivative of the curve $t \mapsto P_t$ is given by
\begin{equation}\label{wh}
\lim_{t_0 \downarrow 0} \dfrac{W_h\big(P_{t+t_0}, P_t\big)}{t_0}
= \sqrt{I_g(P_t|P_{\infty})},~ t\in [0,T);
\end{equation}
see, e.g. Corollary 5.20 and Theorem 5.21 of \cite{DNS}, and Theorem 3.18 of \cite{LW}.

\subsection{The subset $\Lambda_{\var}^g$}

Due to the nonlinearity, it is highly challenging to establish inequalities and convergence rates on the whole space, especially for functionals that are not displacement convex. Therefore, we restrict our analysis to the set
\begin{equation}
  \Lambda_{\var}^g=\big\{ P=p\maL^d \big| \|g(p)-g(p_{\infty}) \|_{L^{\infty}} \leq \var \big\}.
\end{equation}

\begin{lem}\label{Ag}
Assume that Assumption \ref{reg} holds and that $P_0\in \Lambda_{\var}^g$. Then, for all $t\in [0, T]$, the pathwise strong solution of \eqref{MVSDE} satisfies
\begin{equation}\label{PA}
  P_t\in \Lambda_{\var}^g.
\end{equation}
Furthermore, this implies
\begin{equation}\label{bound}
m_{\var}p_{\infty} \leq p_t \leq  M_{\var} p_{\infty},
\end{equation}
where $m_{\var}:= \rme^{-\frac{b_1\var}{\gamma_2}}$ and $M_{\var}:=\rme^{\frac{b_2\var}{\gamma_1}}$.
\end{lem}
\begin{proof}
Since $g$ is smooth and increasing, we have
\begin{equation*}
  \Lambda_{\var}^g=\big\{ P=p\maL^d \big| p_{\infty, c-\var} \leq p \leq p_{\infty, c+\var},~ p_{\infty, r}:=g^{-1}(-\Phi+r) \big\}.
\end{equation*}
Denote $\phi_t:=g(p_t)-g(p_{\infty})$. Then $\nabla \phi_t=g'(p_t)\nabla p_t+\nabla \Phi$ and
\begin{equation*}
  \partial_t \phi_t=g'(p_t) \partial_t p_t =g'(p_t) {\rm div}(h(p_t) \nabla \phi_t)
  =g'(p_t)h(p_t)\Delta \phi_t +g'(p_t) h'(p_t)\nabla p_t \cdot \nabla \phi_t .
\end{equation*}
If \eqref{PA} does not hold, then there exists a time $t_*$ such that $\|\phi_{t_{*}}\|_{L^{\infty}} > \var$.
Without loss of generality, let $t_*$ be the first time at which the maximum value is attained, and suppose that this maximum is achieved at $x_*$. Thus, $\phi_{t_*}(x_*)>\var$.
Using $\nabla \phi_{t_*}(x_*)=0$ and $\Delta \phi_{t_*}(x_*)\leq 0$, we derive
$$\partial_t \phi_t(x)\big|_{t=t_*, x=x_*} =g'(p_{t_*}(x_*))h(p_{t_*}(x_*))\Delta \phi_{t_*}(x_*) \leq 0, $$
which contradicts the choice of $t_*$ and $x_*$.
Therefore, we have
$$\phi_t(x)\leq \var, ~t\in [0,T],~ x\in \R^d. $$
By the same argument, we obtain
$$\phi_t(x)\geq -\var, ~t\in [0,T],~ x\in \R^d. $$
Consequently,
$$\|\phi_t \|_{L^{\infty}}\leq \var,$$
which completes the proof of \eqref{PA}.

Denote $a(r):=rg'(r)=\frac{f'(r)}{b(r)}$. Assumption \ref{reg} ensures that $a_1:=\frac{\gamma_1}{b_2}\leq a(r) \leq \frac{\gamma_2}{b_1}=:a_2$.
A direct calculation shows that
\begin{equation*}
\begin{aligned}
g(p_{\infty} \rme^{\frac{\var}{a_1}})= & \int_1^{p_{\infty}} \dfrac{f'(r)}{rb(r)} \rmd r  + \int_{p_{\infty}}^{p_{\infty} \rme^{\frac{\var}{a_1}}} \dfrac{f'(r)}{rb(r)} \rmd r \\
\geq & g(p_{\infty}) +a_1 \int_1^{\rme^{\frac{\var}{a_1}}} \dfrac{1}{r} \rmd r \\
= & g(p_{\infty})+\var
\end{aligned}
\end{equation*}
and
\begin{equation*}
\begin{aligned}
g(p_{\infty} \rme^{-\frac{\var}{a_2}})= & \int_1^{p_{\infty}} \dfrac{f'(r)}{rb(r)} \rmd r  + \int_{p_{\infty}}^{p_{\infty} \rme^{-\frac{\var}{a_2}}} \dfrac{f'(r)}{rb(r)} \rmd r \\
\leq & g(p_{\infty}) +a_2 \int_1^{\rme^{-\frac{\var}{a_2}}} \dfrac{1}{r} \rmd r \\
= & g(p_{\infty})-\var.
\end{aligned}
\end{equation*}
Thus, \eqref{PA} implies that
\begin{equation*}
g(p_{\infty} \rme^{-\frac{\var}{a_2}})\leq g(p_{\infty})-\var \leq g(p_t) \leq g(p_{\infty})+\var \leq  g(p_{\infty} \rme^{\frac{\var}{a_1}}).
\end{equation*}
Using the smoothness and monotonicity of $g$, we obtain \eqref{bound}.
\end{proof}

\begin{rem}\label{AWh}
The set $\Lambda_{\var}^g$ cannot be replaced by $B_{\delta}^{W_h}$, a ball in the $W_h$ metric. Indeed, there exists a counterexample showing that $B_{\delta}^{W_h} \subset \Lambda_{\var}^g$ does not hold for any $\delta >0$, even in a linear case.
Let $f(r)=r$, $b(r)=1$, and $\Phi(x)=\frac{|x|^2}{2}$. Then $g(r)=\ln r$, $p_{\infty}(x)=(2\pi)^{-\frac{d}{2}} \rme ^{-\frac{|x|^2}{2}}$ and
$$\Lambda_{\var}^g=\big\{ P=p\maL^d \big| \rme ^{-\var} p_{\infty} \leq p \leq \rme^{\var} p_{\infty} \big\}.$$

Define the translated density $p_{\infty}^L(x):=p_{\infty}(x -L\vec{a})$, where $L>0$ and $\vec{a}$ is a prescribed direction.
Set the initial value as $p^L_0:=(1-\var_L)p_{\infty} +\var_L p_{\infty}^L$, where $\var_L=\frac{\delta^2}{4L^2}$.
Then
\begin{equation*}
  W_h(p^L_0, p_{\infty}) =W_2(p^L_0, p_{\infty}) \leq \sqrt{ \var_L }W_2(p_{\infty}^L, p_{\infty})=\frac{\delta}{2}.
\end{equation*}

According to the properties of the Ornstein-Uhlenbeck semigroup (see, for instance, Section 2.7.1 of \cite{BGL}), the density function at $t\in [0,T]$ is given by
\begin{equation*}
p^L_t=(1-\var_L)p_{\infty} +\var_L p_{\infty}^{\rme^{-tL}}
\end{equation*}
and
\begin{equation*}
\begin{aligned}
  \dfrac{p_t^L}{p_{\infty}} =&1-\var_L +\var_L \rme^{-\frac{|x-\rme^{-t}L\vec{a}|^2}{2}+\frac{|x|^2}{2}} \\
  =&1-\var_L +\var_L \rme^{Lx\cdot\vec{a}\rme^{-t}-\frac{1}{2}\rme^{-2t}L^2} \rightarrow +\infty, ~ \text{ as } x\cdot\vec{a}\rightarrow +\infty,
\end{aligned}
\end{equation*}
which implies that
$p_t^L \maL^d \notin \Lambda_{\var}^g$ for any $\var>0$.
\end{rem}

\section{Mobility dependent drift and nonlinear diffusion}\label{mainnon}

Lemma \ref{Ag} shows that the solutions with initial condition in $\Lambda_{\var}^g$ remain in this set.
In what follows, we restrict our analysis to the set $\Lambda_{\var}^g$ and establish the corresponding inequalities.

\subsection{Logarithmic Sobolev inequality and convergence rate of the free energy functional}\label{LSI}
We now establish the logarithmic Sobolev inequality, which, together with \eqref{rela}, yields the convergence rate of the relative entropy $H_g(P_t|P_{\infty})$.
As in \cite{CCF}, we employ the Polyak-${\L}$ojasiewicz inequality to obtain the desired result.

\begin{thm}[Logarithmic Sobolev inequality]\label{HI}
  Suppose that Assumption \ref{reg} holds. Then, for any $P_0\in \Lambda_{\var}^g$,
  \begin{equation}\label{HIineq}
    H_g(P_t|P_{\infty}) \leq C_{HI} I_g(P_t|P_{\infty}),~ t\in[0,T),
  \end{equation}
  where $C_{HI}:=\dfrac{C_HC_{PL}}{C_I}$. The constants $C_H$, $C_I$ and $C_{PL}$ are defined in \eqref{CH}, \eqref{CI} and \eqref{C3}.
\end{thm}

\begin{proof}
Recalling
$$\phi_t=g(p_t)-g(p_{\infty}), \|\phi_t\|_{L^{\infty}}\leq \var \text{ and } \partial_t \phi_t=g'(p_t)\rmdiv(h(p_t)\nabla \phi_t),$$
we rewrite $p_t$ as $$p_t=g^{-1}(g(p_{\infty})+\phi_t).$$

We next prove the theorem in four steps, using the Poincar\'e inequality and the Polyak-${\L}$ojasiewicz inequality.

\textbf{Step 1.}
Estimate of relative entropy
\begin{equation}\label{Hineq}
  H_g(P_t|P_{\infty}) \leq C_H \int_{\R^d} \phi_t^2 h(p_{\infty}) \rmd x,~ \forall t \in [0,T),
\end{equation}
where
\begin{equation}\label{CH}
C_M:=\sup_{\|\phi\|_{L^{\infty}} \leq \var}\frac{|h'(r)|}{\gamma_1}+\frac{|f''(r)h(r)|}{\gamma_1^2}\Big|_{r=g^{-1}(g(p_{\infty})+\phi)},~ C_H:=\Big(\frac{1}{2\gamma_1}+\frac{C_M \rme^{C_M \var}\var}{3\gamma_1}\Big).
\end{equation}

Denote
$$K(\phi_t) =K(p_t)\big|_{p_t=g^{-1}(g(p_{\infty})+\phi_t)}:=\eta(p_t)-\eta(p_{\infty})-g(p_{\infty})(p_t-p_{\infty})\big|_{p_t=g^{-1}(g(p_{\infty})+\phi_t)}.$$
Then
\begin{equation}\label{Kder}
\begin{aligned}
  \dfrac{\rmd}{\rmd \phi} K(\phi) =& \dfrac{\rmd}{\rmd p}K(p) \cdot \dfrac{\rmd p}{\rmd \phi}\Big|_{p=g^{-1}(g(p_{\infty})+\phi)}\\
  =& (\eta'(p)-g(p_{\infty})) (g^{-1})'(g(p_{\infty})+\phi)|_{p=g^{-1}(g(p_{\infty})+\phi)}
\end{aligned}
\end{equation}
and
$$H_g(P_t|P_{\infty})=\int_{\R^d} K(\phi_t) \rmd x.$$
Let
$A(\phi):=(g^{-1})'(g(p_{\infty})+\phi)$. Then
$\dfrac{\rmd}{\rmd \phi} K(\phi)=\phi A(\phi)$.

According to the inverse function theorem, we obtain
\begin{equation*}
  A(0)=(g^{-1})'(g(p_{\infty}))=\dfrac{1}{g'(p_{\infty})}=\dfrac{h(p_{\infty})}{f'(p_{\infty})}
\end{equation*}
and
\begin{equation*}
A(\phi)=\dfrac{h(r)}{f'(r)}\Big|_{r=g^{-1}(g(p_{\infty})+\phi)}.
\end{equation*}
Now we want to show $|A'(\phi_t)|\leq C_M \rme^{C_M \var} A(0)$.
From the estimate $\|\phi_t\|_{L^{\infty}}\leq \var$, we have
\begin{equation}\label{claim1}
\begin{aligned}
  \Big|\dfrac{\rmd}{\rmd \phi_t}\ln A(\phi_t)\Big|
  =&\Big| \dfrac{\rmd}{\rmd r} \ln \dfrac{h(r)}{f'(r)} \cdot \dfrac{\rmd r}{\rmd \phi}\Big|_{r=g^{-1}(g(p_{\infty})+\phi_t)} \Big| \\
  =&\Big| \dfrac{f'(r)}{h(r)} \Big( \dfrac{h'(r)f'(r)-h(r)f''(r)}{(f'(r))^2} \Big)\cdot \dfrac{1}{g'(r)}\Big|_{r=g^{-1}(g(p_{\infty})+\phi_t)} \Big| \\
  \leq &\sup_{\|\phi_t\|_{L^{\infty}} \leq \var}\dfrac{|h'(r)|}{\gamma_1}+\dfrac{|h(r) f''(r)|}{\gamma_1^2}\Big|_{r=g^{-1}(g(p_{\infty})+\phi_t)}.
\end{aligned}
\end{equation}
Since $g$, $h$ and $f$ are smooth functions and $g$ is strictly increasing, we have
$$C_M=\sup_{\|\phi\|_{L^{\infty}} \leq \var}\frac{|h'(r)|}{\gamma_1}+\frac{|f''(r)h(r)|}{\gamma_1^2}\Big|_{r=g^{-1}(g(p_{\infty})+\phi)}< +\infty.$$
Consequently,
\begin{equation}
\rme^{-C_M\var}A(0)\leq A(\phi_t) \leq \rme^{C_M\var} A(0).
\end{equation}
Combining $A'(\phi)=A(\phi)\cdot\frac{\rmd}{\rmd \phi}\ln A(\phi)$, we obtain
\begin{equation}\label{equa2}
  |A'(\phi_t)|\leq C_M |A(\phi_t)|\leq C_M \rme^{C_M \var}A(0)=\dfrac{ C_M \rme^{C_M \var}}{g'(p_{\infty})}.
\end{equation}

The integral $K(\phi)=\int_0^{\phi} y A(y) \rmd y$ implies that
$$H_g(P_t|P_{\infty})=\int_{\R^d} \int_0^{\phi_t} y A(y) \rmd y \rmd x.$$
By Taylor's formula
\begin{equation*}
A(y)=A(0)+\int_0^y A'(r) \rmd r =\dfrac{1}{g'(p_{\infty})} +\int_0^y A'(r) \rmd r,
\end{equation*}
we have
\begin{equation*}
  H_g(P_t|P_{\infty})=\int_{\R^d} \int_0^{\phi_t}  \dfrac{y}{g'(p_{\infty})}\rmd y \rmd x +\int_{\R^d} \int_0^{\phi_t} y \int_0^y A'(r) \rmd r \rmd y \rmd x.
\end{equation*}

Owing to \eqref{equa2} and $\|\phi_t\|_{L^{\infty}}\leq \var$, we obtain
\begin{equation*}
\begin{aligned}
& \Big|H_g(P_t|P_{\infty})-\dfrac{1}{2}\int_{\R^d} \dfrac{1}{g'(p_{\infty})} \phi_t^2 \rmd x\Big|  \\
\leq & \int_{\R^d} \int_0^{\phi_t} \dfrac{C_M \rme^{C_M \var}}{g'(p_{\infty})} y^2 \rmd y \rmd x
\leq  \dfrac{C_M \rme^{C_M \var} \var}{3} \int_{\R^d} \dfrac{\phi_t^2}{g'(p_{\infty})} \rmd x .
\end{aligned}
\end{equation*}
It follows that
\begin{equation*}
\begin{aligned}
 H_g(P_t|P_{\infty})
 \leq & \dfrac{1}{2} \int_{\R^d} \dfrac{h(p_{\infty})}{f'(p_{\infty})} \phi_t^2 \rmd x +\dfrac{C_M \rme^{C_M \var} \var}{3} \int_{\R^d} \int_{\R^d} \dfrac{h(p_{\infty})}{f'(p_{\infty})} \phi_t^2 \rmd x \\
 \leq & \Big(\dfrac{1}{2\gamma_1}+\dfrac{C_M \rme^{C_M \var} \var}{3\gamma_1}\Big) \int_{\R^d} \phi_t^2 h(p_{\infty}) \rmd x \\
 = & C_H \int_{\R^d} \phi_t^2 h(p_{\infty}).
 \end{aligned}
\end{equation*}

\textbf{Step 2.}
Estimate of Fisher information
\begin{equation}\label{Iineq}
  I_g(P_t|P_{\infty})\geq C_I\int_{\R^{d}} |\nabla \phi_t|^2 h(p_{\infty}) \rmd x,
\end{equation}
where
\begin{equation}\label{CI}
C_M':=\sup_{\|\phi\|_{L^{\infty}} \leq \var}\frac{|h'(r)|}{\gamma_1}\Big|_{r=g^{-1}(g(p_{\infty})+\phi)},~C_I:=C_M'\rme^{-C_M'\var}.
\end{equation}

Let us denote
$$\Pi(\phi):=h(r)\big|_{r=g^{-1}(g(p_{\infty})+\phi)}.$$
A direct calculation yields $\Pi(0)=h(p_{\infty})$.
Similar to \eqref{claim1} , we obtain
\begin{equation*}
\begin{aligned}
  \Big| \dfrac{\rmd }{\rmd \phi_t} \ln \Pi(\phi_t) \Big|
  =&\Big| \dfrac{\rmd }{\rmd r} \ln h(r) \cdot \dfrac{\rmd r}{\rmd \phi}\Big|_{r=g^{-1}(g(p_{\infty})+\phi_t)} \Big| \\
  =&\Big| \dfrac{h'(r)}{h(r)}\cdot \dfrac{1}{g'(r)}\Big|_{r=g^{-1}(g(p_{\infty})+\phi_t)} \Big| \\
  \leq& \sup_{\|\phi_t\|_{L^{\infty}} \leq \var}\dfrac{|h'(r)|}{\gamma_1}\Big|_{r=g^{-1}(g(p_{\infty})+\phi_t)}.
\end{aligned}
\end{equation*}
The smoothness of $h$ and $g$, together with the monotonicity of $g$, implies that
$$C_M'=\sup_{\|\phi\|_{L^{\infty}} \leq \var}\frac{|h'(r)|}{\gamma_1}\Big|_{r=g^{-1}(g(p_{\infty})+\phi)}< +\infty.$$
Therefore, we obtain
\begin{equation*}
C_M'\rme^{-C_M'\var}\Pi(0)\leq \Pi(\phi_t) \leq C_M'\rme^{C_M' \var}\Pi(0),
\end{equation*}
which means
$$C_M'\rme^{-C_M'\var} h(p_{\infty}) \leq h(p_t) \leq C_M'\rme^{C_M' \var} h(p_{\infty}).$$
It follows that
\begin{equation*}
  I_g(P_t|P_{\infty})=\int_{\R^d} |\nabla \phi_t|^2 h(p_t)\rmd x \geq C_I \int_{\R^d} |\nabla \phi_t|^2 h(p_{\infty})\rmd x.
\end{equation*}

\textbf{Step 3.}
Using the Poincar\'e inequality to derive the Polyak-${\L}$ojasiewicz inequality
\begin{equation}\label{add3}
\int_{\R^d} \phi_t^2 h(p_{\infty}) \rmd x \leq C_{PL} \int_{\R^d} |\nabla \phi_t|^2 h(p_{\infty}) \rmd x,
\end{equation}
where
\begin{equation}\label{C3}
C_{PL}:=\Big( 1+ \frac{\gamma_2 b_2^3 \int_{\R^d} (1+M_{\var}) \rmd \mu}{\gamma_1^2 b_1^3} \Big)^2 C_P.
\end{equation}

Using Lagrange mean value theorem, there exists a function $\bar{\xi}:\R^d \rightarrow [0,1]$ such that
\begin{equation*}
  \phi_t=g(p_t)-g(p_{\infty})=g'\big(\bar{\xi} p_{\infty}+(1-\bar{\xi}) p_t\big)(p_t-p_{\infty}).
\end{equation*}
We denote $\xi_t:=\bar{\xi} p_{\infty}+(1-\bar{\xi}) p_t$.
Then
\begin{equation}\label{equa3}
  (\bar{\xi} +(1-\bar{\xi})m_{\var}) p_{\infty} \leq \xi_t \leq (\bar{\xi} +(1-\bar{\xi})M_{\var}) p_{\infty}.
\end{equation}
In the case $g(p_t)=g(p_{\infty})$, the proof is trivial. From now on, we assume that $\phi_t\ne0$.
Then we have
$$g'(\xi_t)=\frac{\phi_t}{p_t-p_{\infty}}.$$
Let
\begin{equation*}
\tilde{g}(r):=
 \frac{1}{g'(r)h(p_{\infty})} 
\end{equation*}
and
\begin{equation*}
  \bar{\phi}_t:=\int_{\R^d} \phi_t \rmd \mu.
\end{equation*}
Thus,
\begin{equation*}
\begin{aligned}
  0 =&\int_{\R^d} p_t-p_{\infty} \rmd x
  =\int_{\R^d} \dfrac{\phi_t}{g'(\xi_t)} \rmd x \\
  =&\int_{\R^d} \dfrac{\phi_t}{g'(\xi_t) h(p_{\infty})} \rmd \mu \\
  =&\int_{\R^d} \tilde{g}(\xi_t)(\phi_t-\bar{\phi}_t) \rmd \mu + \int_{\R^d} \tilde{g}(\xi_t)\bar{\phi} \rmd \mu,
\end{aligned}
\end{equation*}
which implies
\begin{equation}\label{add2}
\begin{aligned}
  \Big|\int_{\R^d} \tilde{g}(\xi_t) \rmd \mu\Big| \cdot\Big| \bar{\phi_t}\Big|=\Big|\langle \phi_t-\bar{\phi}_t,  \tilde{g}(\xi_t) \rangle_{L^2(\mu)}\Big|.
\end{aligned}
\end{equation}
It can be concluded through
\begin{equation*}
 \dfrac{h(\xi_t)}{h(p_{\infty})} \leq \dfrac{b_2 \xi_t}{b_1 p_{\infty}} \leq \dfrac{b_2 \bar{\xi}}{b_1}+ \dfrac{b_2(1-\bar{\xi})M_{\var}}{b_1} < +\infty
\end{equation*}
that
\begin{equation}
\| \tilde{g}(\xi_t) \|_{L^2(\mu)}^2 =\int_{\R^d} \dfrac{(h(\xi_t))^2}{(f'(\xi_t))^2 (h(p_{\infty}))^2} \rmd \mu
\leq \int_{\R^d} \dfrac{b_2^2}{\gamma_1^2 b_1^2}(\bar{\xi} +(1-\bar{\xi})M_{\var}) \rmd \mu
< +\infty
\end{equation}
and
\begin{equation*}
  \Big|\int_{\R^d} \tilde{g}(\xi_t) \rmd \mu\Big| =\Big|\int_{\R^d} \dfrac{h(\xi_t)}{f'(\xi_t) h(p_{\infty}) } \rmd x\Big|
  \geq \int_{\R^d} \dfrac{b_1}{\gamma_2} \xi_t \rmd x
  = \dfrac{b_1}{\gamma_2} \ne 0.
\end{equation*}
By the Minkowski inequality, \eqref{add2} and the H\"older inequality, we deduce
\begin{equation*}
\begin{aligned}
  \| \phi_t \|_{L^2(\mu)} \leq & \| \phi_t-\bar{\phi}_t \|_{L^2(\mu)} +\| \bar{\phi}_t \|_{L^2(\mu)} \\
  \leq & \| \phi_t-\bar{\phi}_t \|_{L^2(\mu)} + |\bar{\phi}_t|\cdot\int_{\R^d}1\rmd \mu \\
  \leq & \| \phi_t-\bar{\phi}_t \|_{L^2(\mu)} + b_2 \dfrac{| \langle \phi_t-\bar{\phi}_t , \tilde{g} (\xi_t) \rangle_{L^2(\mu)} |}{|\int_{\R^d} \tilde{g}(\xi_t) \rmd \mu|}\\
  \leq & \| \phi_t-\bar{\phi}_t \|_{L^2(\mu)} + b_2 \dfrac{\| \tilde{g}(\xi_t) \|_{L^2(\mu)}}{|\int_{\R^d} \tilde{g}(\xi_t) \rmd \mu|}\| \phi_t-\bar{\phi}_t \|_{L^2(\mu)},
\end{aligned}
\end{equation*}
which implies
\begin{equation*}
  \int_{\R^d} \phi_t^2 \rmd \mu \leq \Big( 1+ b_2 \dfrac{\| \tilde{g}(\xi_t) \|_{L^2(\mu)}}{|\int_{\R^d} \tilde{g}(\xi_t) \rmd \mu|} \Big)^2 \int_{\R^d} (\phi_t-\bar{\phi}_t)^2 \rmd \mu.
\end{equation*}
Applying the Assumption \ref{reg} (vii), we obtain
\begin{equation*}
 \int_{\R^d} \phi_t^2 \rmd \mu \leq \Big( 1+ b_2 \dfrac{\| \tilde{g}(\xi_t) \|_{L^2(\mu)}}{|\int_{\R^d} \tilde{g}(\xi_t) \rmd \mu|} \Big)^2 C_P \int_{\R^d} |\nabla \phi_t|^2 h(p_{\infty})\rmd x
 \leq C_{PL} \int_{\R^d} |\nabla \phi_t|^2 h(p_{\infty})\rmd x.
\end{equation*}

\textbf{Step 4.} The final estimate
$$H_g(P_t|P_{\infty}) \leq C_{HI} I_g(P_t|P_{\infty}).$$

In view of \eqref{Hineq}, \eqref{add3} and \eqref{Iineq},
we conclude that
\begin{equation*}
\begin{aligned}
  H_g(P_t|P_{\infty}) \leq&  C_H \int_{\R^d} \phi_t^2 h(p_{\infty}) \rmd x\\
  \leq &  C_H C_{PL} \int_{\R^d} |\nabla \phi_t|^2 h(p_{\infty})\rmd x \\
  \leq & \dfrac{C_H C_{PL}}{C_I} I_g(P_t|P_{\infty})
  =C_{HI} I_g(P_t|P_{\infty}).
\end{aligned}
\end{equation*}

\end{proof}

\begin{rem}
In fact, one can also directly use \eqref{bound} to derive the estimate $H_g(P_t|P_{\infty}) \leq \frac{\gamma_2 M_{\var}^2b_2^2}{2 b_1^2 m_{\var} \gamma_1^2}C_P'I_g(P_t|P_{\infty})$, where $C_P'$ is the Poincar\'e constant for $P_{\infty}$. However, this estimate is simpler but less precise. In contrast, our estimate takes into account the influence of the nonlinear term $h$, making it more accurate.
\end{rem}

\begin{cor}[Convergence rate of the relative entropy]
  Suppose that Assumption \ref{reg} holds. Then, for any $P_0\in \Lambda_{\var}^g$, we have the convergence rate of the relative entropy
  \begin{equation}\label{Hconver}
    H_g(P_t|P_{\infty}) \leq \rme^{-\frac{t}{C_{HI}}}H_g(P_0|P_{\infty}), ~ t\in[0,T)
  \end{equation}
  and the convergence rate of the free energy
  \begin{equation}\label{freeconver}
  \mathcal{F}(p_t)-\mathcal{F}(p_{\infty})\leq \rme^{-\frac{t}{C_{HI}}} (\mathcal{F}(p_0)-\mathcal{F}(p_{\infty})).
  \end{equation}
\end{cor}

\begin{proof}
  Combining \eqref{rela} and \eqref{HIineq}, we obtain
  \begin{equation*}
    \dfrac{\rmd}{\rmd t}H_g(P_t|P_{\infty})=-I_g(P_t|P_{\infty}) \leq -\dfrac{1}{C_{HI}} H_g(P_t|P_{\infty}).
  \end{equation*}
The Gronwall inequality implies \eqref{Hconver}.
It follows from \eqref{Hg} that \eqref{freeconver} also holds.
\end{proof}

\begin{cor}[Talagrand inequality]
  Suppose that Assumptions \ref{reg} and \ref{met} hold. Then, for any $P_0\in \Lambda_{\var}^g$, we have the Talagrand inequality
  \begin{equation}\label{Tineq}
    W_h^2(P_t, P_{\infty}) \leq C_T H_g(P_t|P_{\infty}), ~ t\in[0,T),
  \end{equation}
  where $C_T:=4C_{HI}$.
  Moreover,
  \begin{equation}\label{WH}
   W_h^2(P_t, P_{\infty}) \leq C_T \rme^{-\frac{t}{C_{HI}}}H_g(P_0|P_{\infty}), ~ t\in[0,T).
  \end{equation}
\end{cor}

\begin{proof}
By the definition of $W_h$ with \eqref{HIineq}, we deduce that, for $t_0 \geq t$,
\begin{equation*}
\begin{aligned}
  W_h(P_t, P_{t_0}) \leq & \int_t^{t_0} \dfrac{I_g(P_s|P_{\infty})}{\sqrt{I_g(P_s|P_{\infty})}} \rmd s \\
  \leq & \sqrt{C_{HI}}\int_t^{t_0} \dfrac{I_g(P_s|P_{\infty})}{\sqrt{H_g(P_s|P_{\infty})}} \rmd s \\
  \leq & -2\sqrt{C_{HI}}\int_t^{t_0} \dfrac{\rmd}{\rmd s}\sqrt{H_g(P_s|P_{\infty})}\rmd s \\
  \leq & 2\sqrt{C_{HI}} (\sqrt{H_g(P_t|P_{\infty})} -\sqrt{H_g(P_{t_0}|P_{\infty})}).
\end{aligned}
\end{equation*}
Letting $t_0\rightarrow +\infty$, we conclude that
\begin{equation*}
  W_h(P_t, P_{\infty}) \leq 2\sqrt{C_{HI}} \sqrt{H_g(P_t|P_{\infty})},
\end{equation*}
which implies
\begin{equation*}
  W_h^2(P_t, P_{\infty}) \leq  C_{T} H_g(P_t|P_{\infty}) \leq C_T \rme^{-\frac{t}{C_{HI}}}H_g(P_0|P_{\infty}).
\end{equation*}

\end{proof}

\subsection{HWI inequality}\label{mainHWI}
We now derive the HWI inequality by using the relationship among $W_2$, $W_h$ and $\dot{H}^{-1}$ norm.
We first introduce the weighted homogeneous Sobolev norm and a seminorm induced by $\dot{H}^{-1}$ (see, for instance, \cite{CCGT, P}).

\begin{de}[Weighted homogeneous Sobolev norm]\label{sobolev}
  For a given positive measure $\pi$, the {\em weighted homogeneous Sobolev norm} of any signed measure $\nu$ is defined by
  \begin{equation*}
    \|\nu\|_{\dot{H}^{-1}(\pi)}:=\sup \Big\{ |\langle \phi, \nu \rangle | \Big| \| \phi\|_{\dot{H}^1(\pi)} \leq 1 \Big\}, \text{ where }  \| \phi\|_{\dot{H}^1(\pi)}:=\int_{\R^d} |\nabla \phi|^2 \rmd \pi.
  \end{equation*}
  When $\pi$ is the Lebesgue measure $\maL^d$, we simply write $| \cdot |_{\dot{H}^{-1}}$.
\end{de}

\begin{prop}\label{normprop}
  For a given positive measure $\pi$ and probability measures $P$ and $Q$ with densities $p$ and $q$, respectively, we define the semi-norm by
  \begin{equation}\label{semi}
    D_{\pi}(P, Q):=\|P-Q\|_{\dot{H}^{-1}(\pi)}.
  \end{equation}
  Moreover, $p-q=-\rmdiv(r \nabla \phi)$ and $\int_{\R^d} |\nabla \phi|^2 r \rmd x\ne 0$ imply
  \begin{equation*}
    D_{\pi}(P, Q)=\Big(\int_{\R^d} |\nabla \phi|^2 r \rmd x\Big)^{\frac{1}{2}}.
  \end{equation*}
\end{prop}

\begin{proof}
  On the one hand, Definition \ref{sobolev} and $p-q=-\rmdiv(r \nabla \phi)$ give that
  \begin{equation*}
  \begin{aligned}
    D_{\pi}(P, Q) =& \sup \Big\{ \int_{\R^d} \psi (p-q) \rmd x \Big| \int_{\R^d} |\nabla \psi|^2 r\rmd x \leq 1 \Big\} \\
    =& \sup \Big\{ \int_{\R^d} \nabla \psi \cdot \nabla \phi r \rmd x \Big| \int_{\R^d} |\nabla \psi|^2 r\rmd x \leq 1 \Big\}.
  \end{aligned}
  \end{equation*}
  Therefore, we have
  \begin{equation*}
  D_{\pi}(P, Q) \leq \Big(\int_{\R^d} |\nabla \phi|^2 r\rmd x\Big)^{\frac{1}{2}} \Big( \int_{\R^d} |\nabla \psi|^2 r\rmd x \Big)^{\frac{1}{2}} \leq \Big(\int_{\R^d} |\nabla \phi|^2 r\rmd x\Big)^{\frac{1}{2}}
  \end{equation*}
  with the help of the H\"older inequality and  $\int_{\R^d} |\nabla \psi|^2 r\rmd x \leq 1$.

On the other hand, let
\begin{equation*}
  \psi:=\dfrac{\phi}{\big(\int_{\R^d} |\nabla \phi|^2 r \rmd x \big)^{\frac{1}{2}}},
\end{equation*}
which satisfies $\int_{\R^d} |\nabla \psi|^2 r\rmd x = 1$.
Then
\begin{equation*}
  \int_{\R^d} \nabla \psi \cdot \nabla \phi r \rmd x =\Big(\int_{\R^d} |\nabla \phi|^2 r\rmd x\Big)^{\frac{1}{2}},
\end{equation*}
which means that
\begin{equation*}
  D_{\pi}(P, Q) \geq \Big(\int_{\R^d} |\nabla \phi|^2 r\rmd x\Big)^{\frac{1}{2}}.
\end{equation*}
The proof is complete.
\end{proof}

We next introduce an HWI-like inequality. Instead of using the $W_h$ metric, we employ the $\dot{H}^{-1}$ norm.

\begin{thm}[HWI-like inequality]\label{HWI1}
  Under Assumption \ref{reg}, for any $P_0\in \Lambda_{\var}^g$, we have
  \begin{equation}\label{HWIlike}
    H_g(P_t|P_{\infty}) \leq C_{HWI}' D_{\hat{h}(P_{\infty})}(P_t, P_{\infty}) \sqrt{I_g(P_t|P_{\infty})},~ t\in[0, T),
  \end{equation}
  where $C_{HWI}':=\dfrac{b_2 \gamma_2 C_H}{b_1 m_{\var}\sqrt{C_I}}$.
\end{thm}

\begin{proof}
As in Step 3 of Theorem \ref{HI}, we consider only the case $\phi_t \neq 0$.
Denote
\begin{equation}\label{w}
  w_t:=\dfrac{p_t-p_{\infty}}{\phi_t h(p_{\infty})}.
\end{equation}
Then
\begin{equation*}
  w_t=\dfrac{p_t-p_{\infty}}{h(p_{\infty})(g(p_t)-g(p_{\infty}))}
  \geq \dfrac{1}{b_2 p_{\infty} g'(\xi_t)}
  \geq \dfrac{\bar{\xi}+(1-\bar{\xi})m_{\var}}{b_2 \xi_t g'(\xi_t)}\geq \dfrac{b_1(\bar{\xi}+(1-\bar{\xi})m_{\var})}{b_2 \gamma_2}\geq C_w,
\end{equation*}
where $C_w:=\frac{b_1 m_{\var} }{b_2 \gamma_2}$.
It follows that
\begin{equation}\label{es1}
  \int_{\R^d} \phi_t^2 h(p_{\infty}) \rmd x \leq \dfrac{1}{C_w}\int_{\R^d} \phi_t (p_t-p_{\infty}) \rmd x.
\end{equation}
Let $\varphi_t$ satisfy $p_t-p_{\infty}=-\rmdiv (h(p_{\infty}) \nabla \varphi_t)$.
From the estimate \eqref{es1} and the H\"older inequality, we obtain
\begin{equation}\label{ineq}
\begin{aligned}
\int_{\R^d} \phi_t^2 h(p_{\infty}) \rmd x \leq&  \dfrac{1}{C_w}\int_{\R^d} \nabla \phi_t \cdot \nabla \varphi_t h(p_{\infty}) \rmd x \\
\leq & \dfrac{1}{C_w} \Big(\int_{\R^d}|\nabla \phi_t|^2 h(p_{\infty}) \rmd x\Big)^{\frac{1}{2}}\Big(\int_{\R^d}|\nabla \varphi_t|^2 h(p_{\infty}) \rmd x\Big)^{\frac{1}{2}}.
\end{aligned}
\end{equation}
Proposition \ref{normprop} yields
\begin{equation}\label{es2}
  D_{\hat{h}(P_{\infty})}(P_t, P_{\infty}) =\Big(\int_{\R^d}|\nabla \varphi_t|^2 h(p_{\infty}) \rmd x\Big)^{\frac{1}{2}}.
\end{equation}
Consequently, we have
\begin{equation*}
\begin{aligned}
 H_g(P_t|P_{\infty}) \leq & C_H \int_{\R^d} \phi_t^2 h(p_{\infty}) \rmd x  \\
\leq & \dfrac{C_H}{C_w} \Big(\int_{\R^d}|\nabla \phi_t|^2 h(p_{\infty}) \rmd x\Big)^{\frac{1}{2}}D_{\hat{h}(P_{\infty})}(P_t, P_{\infty}) \\
\leq & C_{HWI}' \sqrt{I_g(P_t|P_{\infty})}D_{\hat{h}(P_{\infty})}(P_t, P_{\infty})
\end{aligned}
\end{equation*}
by \eqref{Hineq}, \eqref{ineq}, \eqref{es2} and \eqref{Iineq}.
\end{proof}

To obtain the genuine HWI inequality, we need to impose a stronger condition. At this point, it is necessary to rely on the control of $W_2$ over $W_h$.

\begin{thm}[HWI inequality]\label{HWI}
Suppose that Assumptions \ref{reg} and \ref{met} hold, and that $p_{\infty}$ is log-concave. Then
\begin{equation*}
  H_g(P_t|P_{\infty}) \leq C_{HWI} W_h(P_t, P_{\infty}) \sqrt{I_g(P_t|P_{\infty})},~ t\in[0, T),
\end{equation*}
where $C_{HWI}:=\dfrac{\gamma_2^3 b_2^2 M_{\var}^2\sqrt{b_2 M_{\var}}}{2b_1^3 \gamma_1^2 m_{\var}^2\sqrt{b_2 C_M'\rme^{-C_M'\var}}}$.
\end{thm}

\begin{proof}
  A proof similar to that of Theorem \ref{HWI1} yields
  \begin{equation}\label{step1}
  \begin{aligned}
    H_g(P_t|P_{\infty})
    \leq & \dfrac{\gamma_2^3 b_2^2 M_{\var}^2}{2b_1^3 \gamma_1^2 m_{\var}^2} D_{P_{\infty}}(P_t|P_{\infty}) \Big(\int_{\R^d}|\nabla \phi_t|^2 p_{\infty} \rmd x\Big)^{\frac{1}{2}} \\
    \leq & C_{HWI}'' D_{P_{\infty}}(P_t|P_{\infty}) \sqrt{I_g(P_t|P_{\infty})}
  \end{aligned}
  \end{equation}
  by replacing $p_{\infty}$ with $h(p_{\infty})$, where $C_{HWI}'':=\frac{\gamma_2^3 b_2^2 M_{\var}^2}{2b_1^3 \gamma_1^2 m_{\var}^2\sqrt{b_2 C_M'\rme^{-C_M'\var}}}$.
Let $Q_s$, $s\in [0,1]$, be the $W_2$-displacement interpolation from $P_{\infty}$ to $P_t$, with density $q_s$.

We claim that
\begin{equation}\label{equa4}
r_s:=\frac{q_s}{p_{\infty}}\leq M_{\var}.
\end{equation}

Let $T$ be $W_2$-optimal map from $P_{\infty}$ to $P_t$.
Denote $T_s(x):=(1-s)x+sT(x)$. Then $Q_s=(T_s)_{\sharp}P_{\infty}$.
Since $p_{\infty}$ is log-concave, we deduce that
\begin{equation}\label{eq1}
  p_{\infty}(T_s(x)) \geq p_{\infty}(x)^{1-s}p_{\infty}(T(x))^s.
\end{equation}
The Monge-Amp\`ere equation gives
\begin{equation}\label{eq5}
  p_t(T(x))\det DT(x)=p_{\infty}(x)
\end{equation}
and
\begin{equation}\label{eq4}
  q_s(T_s(x))=\dfrac{p_{\infty}(x)}{\det\big((1-s){\rm I_{d\times d}}+sDT(x)\big)}.
\end{equation}
In view of \eqref{eq1} and \eqref{eq4}, we get
\begin{equation*}
\begin{aligned}
  \dfrac{q_s(T_s(x))}{p_{\infty}(T_s(x))} \leq & \dfrac{(p_{\infty}(x))^s}{\det\big((1-s){\rm I_{d\times d}}+sDT(x)\big)p_{\infty}(T(x))^s}.
\end{aligned}
\end{equation*}
It follows that
\begin{equation*}
  \begin{aligned}
  \dfrac{q_s(T_s(x))}{p_{\infty}(T_s(x))}\leq & \dfrac{(p_{\infty}(x))^s}{\det(DT(x))^s p_{\infty}(T(x))^s} \\
  =& \Big(\dfrac{p_t(T(x))}{p_{\infty}(T(x))}\Big)^s \\
  \leq & M_{\var}
  \end{aligned}
\end{equation*}
with the help of $\det((1-s){\rm I_{d\times d}}+sDT(x))\geq \det(DT(x))^s$ and \eqref{eq5},
which completes the proof of claim.

We now define the field $v_s$ along the $W_2$-geodesic. That is,
\begin{equation*}
  \partial_s Q_s+\rmdiv(Q_s v_s)=0
\end{equation*}
and
\begin{equation*}
\int_0^1\int_{\R^d}|v_s|^2 \rmd Q_s \rmd s=W_2^2(P_t, P_{\infty}).
\end{equation*}
For the test function $\psi$, we have
\begin{equation*}
\begin{aligned}
  \int_{\R^d} \rmd (P_t-P_{\infty}) = & \int_0^1 \dfrac{\rmd}{\rmd s} \int_{\R^d} \psi \rmd Q_s \rmd s\\
  =& \int_0^1 \int_{\R^d} \nabla \psi \cdot v_s \rmd Q_s \rmd s \\
  \leq & \Big(\int_0^1 \int_{\R^d} |v_s|^2 \rmd Q_s \rmd s\Big)^{\frac{1}{2}}\Big(\int_0^1 \int_{\R^d} |\nabla \psi|^2 \rmd Q_s \rmd s\Big)^{\frac{1}{2}} \\
  = & W_2(P_t, P_{\infty}) \Big(\int_0^1 \int_{\R^d} |\nabla \psi|^2 r_s\rmd P_{\infty} \rmd s\Big)^{\frac{1}{2}} \\
  \leq & W_2(P_t, P_{\infty}) \sqrt{M_{\var}} \Big( \int_{\R^d} |\nabla \psi|^2\rmd P_{\infty}\Big)^{\frac{1}{2}}
\end{aligned}
\end{equation*}
by  the H\"older inequality and \eqref{equa4}.
Taking the supremum over all test functions $\psi$ that satisfy $\int_{\R^d} |\nabla \psi|^2 \rmd P_{\infty}\leq 1$, we obtain
\begin{equation}\label{step2}
  D_{P_{\infty}}(P_t, P_{\infty}) \leq \sqrt{M_{\var}} W_2(P_t, P_{\infty}).
\end{equation}

For any pair $(\tilde{q}_s, \tilde{v}_s)$ satisfying the continuity equation $$\partial_s \tilde{q}_s+\rmdiv (h(\tilde{q}_s) \tilde{v}_s)=0,$$ we have
$$\partial_s \tilde{q}_s+\rmdiv (\tilde{q}_s v_s)=0,$$
where $v_s:=b(\tilde{q}_s) \tilde{v}_s$.
Therefore,
\begin{equation*}
\int_0^1 \int_{\R^d} |v_s|^2 \tilde{q}_s\rmd x\rmd s \leq \int_0^1 \int_{\R^d} |\tilde{v}_s|^2 b(\tilde{q}_s)^2 \tilde{q}_s \rmd x\rmd s
\leq b_2 \int_0^1 \int_{\R^d} |\tilde{v}_s|^2 h(\tilde{q}_s) \rmd x\rmd s.
\end{equation*}
Taking the supremum over all pairs, we get
\begin{equation}\label{step3}
  W_2(P_t, P_{\infty}) \leq \sqrt{b_2} W_h(P_t, P_{\infty}).
\end{equation}

From \eqref{step1}, \eqref{step2} and \eqref{step3}, we can deduce that
\begin{equation*}
\begin{aligned}
  H_g(P_t|P_{\infty}) \leq & C_{HWI}'' D_{P_{\infty}}(P_t|P_{\infty}) \sqrt{I_g(P_t|P_{\infty})} \\
  \leq & C_{HWI}'' \sqrt{b_2 M_{\var}} W_h(P_t, P_{\infty})\sqrt{I_g(P_t|P_{\infty})}.
\end{aligned}
\end{equation*}
\end{proof}

\section{Mobility dependent drift and linear diffusion}\label{lin}
In this section, we consider the case of linear diffusion, that is, $f(r)=\sigma^2 r$. This choice ensures the applicability of the Girsanov theorem.
We now impose an additional conditions on $\Phi$.
Under this condition, using a regularization result as in \cite{MR}, we obtain the convergence rate in the $W_h$ metric.
\begin{ass}[Additional conditions for the potential $\Phi$]\label{Phicon}
The smooth function $\Phi: \R^d \rightarrow [0, +\infty)$ satisfies that there exist constants $C_1, C_2\geq 0$ such that
\begin{equation}\label{Phi1}
-C_1 {\rm I_{d\times d}} \leq D^2 \Phi(x) \leq C_2 {\rm I_{d\times d}}.
\end{equation}
\end{ass}

Fix a time $\tau\in (0,T)$.
The monotonicity of $g$, together with $g(p_{\infty}(x))+\Phi(x)=c_0$, ensures $g(p_{\infty, \max} )+\Phi_{\min}=c_0$.
Combining
$$g(p_{\infty, \max})-g(p_{\infty}(x))=c_0-\Phi_{\min}-\big(c_0-\Phi(x) \big)=\Phi(x)-\Phi_{\min}$$ and
$$g(p_{\infty, \max})-g(p_{\infty}(x)) =\int_{p_{\infty}(x)}^{p_{\infty, \max}} g'(r)\rmd r \leq \dfrac{\gamma_2}{b_1} \ln\frac{p_{\infty, \max}}{p_{\infty}(x)},$$
we conclude that, for all $x\in \R^d$,
$$p_{\infty}(x)\leq p_{\infty, \max} \rme^{-\frac{b_1}{\gamma_2}(\Phi(x)-\Phi_{\min})}.$$ 
The inequality \eqref{Phi1} implies
\begin{equation}\label{Phi4}
|\Delta \Phi|\leq \max\{C_1, C_2\}d=:C_{\Phi}d
\end{equation}
and
\begin{equation*}\label{Phi3}
\begin{aligned}
\Phi(x)=& \Phi(0) +\langle \nabla \Phi(0), x \rangle +\int_0^1(1-t) x \cdot D^2 \Phi(x) \cdot x \rmd t \\
\leq & \Phi(0)+\dfrac{1}{2} |\nabla \Phi(0) |^2 + \dfrac{1}{2} |x|^2+\dfrac{C_2}{2}|x|^2 \\
\leq & C_3|x|^2+C_4,
\end{aligned}
\end{equation*}
where $C_3:=\Phi(0)+\frac{1}{2} |\nabla \Phi(0) |^2>0$ and $C_4:=\frac{1}{2} +\frac{C_2}{2} >0$.
It follows that
\begin{equation*}
\begin{aligned}
s_1:=& \sup_{x\in \R^d} |\nabla \Phi(x)|p_{\infty}(x) \\
\leq& \sup_{|x|\leq R} |\nabla \Phi(x)|p_{\infty}(x) + C p_{\infty, \max} \rme^{ \frac{b_1}{\gamma_2}(\Phi_{\min}+C_4) } \sup_{|x|>R} |x|\rme^{-C_3|x|^2} < +\infty,
\end{aligned}
\end{equation*}
\begin{equation*}
\begin{aligned}
s_2:=& \sup_{x\in \R^d} |\nabla \Phi(x)|^2p_{\infty}(x) \\
\leq& \sup_{|x|\leq R} |\nabla \Phi(x)|^2p_{\infty}(x)+ C^2 p_{\infty, \max} \rme^{ \frac{b_1}{\gamma_2}(\Phi_{\min}+C_4) } \sup_{|x|>R} |x|^2 \rme^{-C_3|x|^2} < +\infty
\end{aligned}
\end{equation*}
by Assumption \ref{reg} (iii).
Recalling that $b$, $h$ are smooth functions and $p\leq M_{\var} p_{\infty} \leq M_{\var}\|p_{\infty}\|_{L^{\infty}}$, we get
$$L_b^{\var}:=\sup_{|r|\leq M_{\var}\|p_{\infty}\|_{L^{\infty}} }|b'(r)|<+\infty, ~ L_{h, 2}^{\var}:= \sup_{|r|\leq M_{\var}\|p_{\infty}\|_{L^{\infty}} } |h''(r)|<+\infty. $$

\begin{lem}\label{lemW}
Suppose that Assumptions \ref{reg}, \ref{met} and \ref{Phicon} hold. Then, for any $P_0\in \Lambda_{\var}^g$, we have
\begin{equation}\label{L2W2}
  \int_0^{\tau}\|p_t-p_{\infty}\|_{L^2}^2 \rmd t \leq C_{LW}W_2^2(P_0, P_{\infty}),
\end{equation}
where $C_{\lambda}:= -\dfrac{b(0) C_{\Phi}d}{2} +\dfrac{(L_{h,2}^{\var} M_{\var} s_1)^2}{2\sigma^2} $,  $C_{H,\var}=\dfrac{2(\|p_0\|_{L^{\infty}}^{\frac{1}{2}} -\|p_{\infty}\|_{L^{\infty}}^{\frac{1}{2}} )}{\ln( \|p_0\|_{L^{\infty}}/\|p_{\infty}\|_{L^{\infty}}) )}$, $C_{LW}:=\dfrac{\rme^{2C_{\lambda}\tau}}{\sigma^2}C_{H,\var}^2$.
\end{lem}

\begin{proof}
Denote $r_t:=p_t-p_{\infty}$. Then
\begin{equation}\label{NFPE2}
\partial_t r_t=\sigma^2 \Delta r_t +\rmdiv \big(\nabla \Phi (h(p_t)-h(p_{\infty})) \big)
\end{equation}
and
$$h(p_t)-h(p_{\infty})=(p_t-p_{\infty}) \int_0^1 h'(p_{\infty}+\theta(p_t-p_{\infty}))\rmd \theta.$$
Define
$$J _t:=\int_0^1 h'(p_{\infty}+\theta(p_t-p_{\infty}))\rmd \theta.$$
We then obtain
$$h(p_t)-h(p_{\infty})=J_t r_t,$$
and the equation \eqref{NFPE2} can be rewrite as
\begin{equation}\label{NFPE3}
  \partial_t r_t=\sigma^2 \Delta r_t+ \beta \nabla\cdot (\nabla \Phi r_t)+\nabla\cdot (\nabla \Phi(J_t-\beta)r_t),
\end{equation}
where $\beta:=h'(0)=b(0)>0$.
For the second term on the right hand side, estimate $0\leq p_{\infty}+\theta(p_t-p_{\infty}) \leq M_{\var}p_{\infty}$ implies
\begin{equation}\label{equa5}
\begin{aligned}
  |\nabla \Phi(x)|\cdot |J_t-\beta|\leq & |\nabla \Phi(x)|\cdot \Big|\int_0^1 h'\big(p_{\infty}(x)+\theta(p_t(x)-p_{\infty}(x))\big)-h'(0)\rmd \theta\Big| \\
  \leq & |\nabla \Phi(x)| \cdot \int_0^1 \int_0^{M_{\var} p_{\infty}(x)} |h''(s)|\rmd s \rmd \theta \\
  \leq & |\nabla \Phi(x)|\cdot L_{h,2}^{\var} M_{\var} p_{\infty}(x) \\
  \leq &  L_{h,2}^{\var} M_{\var} s_1=: K_{\var} <+\infty.
\end{aligned}
\end{equation}

  Let $\psi_t$ satisfy $-\Delta \psi_t=r_t$ and set $\lambda_t:=\|r_t\|_{\dot{H}^{-1}}^2=\int_{\R^d} |\nabla \psi_t|^2 \rmd x =\int_{\R^d} \psi_t r_t \rmd x \geq 0 $.
  Then
  \begin{equation*}
  \begin{aligned}
  \dfrac{1}{2}\dfrac{\rmd}{\rmd t}\lambda_t=&\dfrac{1}{2}\int_{\R^d} \partial_t\psi_t\cdot r_t+ \psi_t\cdot \partial_t r_t \rmd x  \\
  =&\dfrac{1}{2}\int_{\R^d} -\partial_t\psi_t\cdot \Delta \psi_t- \psi_t\cdot \rmdiv(\nabla\partial_t \psi_t) \rmd x \\
  =&-\int_{\R^d} \psi_t\cdot \Delta(\partial_t \psi_t) \rmd x \\
  =&\sigma^2 \int_{\R^d} \psi_t \Delta r_t \rmd x +\beta \int_{\R^d} \psi_t \nabla\cdot(\nabla \Phi r_t) \rmd x +\int_{\R^d} \psi_t\nabla\cdot(\nabla \Phi (J_t -\beta)r_t)\rmd x .
  \end{aligned}
  \end{equation*}
  Denote
  \begin{equation*}
  \begin{aligned}
  {\rm (I)}:= & \sigma^2 \int_{\R^d} \psi_t \Delta r_t \rmd x, \\
  {\rm (II)}:= & \beta \int_{\R^d} \psi_t \nabla\cdot(\nabla \Phi r_t) \rmd x, \\
  {\rm (III)}:=& \int_{\R^d} \psi_t\nabla\cdot(\nabla \Phi (J_t-\beta)r_t)\rmd x.
  \end{aligned}
  \end{equation*}
  A direct calculation yields
  \begin{equation}\label{I}
  {\rm (I)}=-\sigma^2 \int_{\R^d} r_t^2 \rmd x
  \end{equation}
  and
  \begin{equation*}
  \begin{aligned}
  {\rm (II)}
  =& \beta\int_{\R^d} \nabla \psi_t\cdot \nabla \Phi \Delta \psi_t \rmd x \\
  =& -\beta \int_{\R^d} \nabla(\nabla \psi_t \cdot \nabla \Phi) \cdot \nabla \psi_t \rmd x \\
  =& -\beta \Big( \int_{\R^d} \nabla \psi_t \cdot D^2\psi_t \cdot \nabla \Phi \rmd x+ \int_{\R^d} \nabla \psi_t \cdot D^2 \Phi \cdot \nabla \psi_t \rmd x  \Big ) \\
  =& -\beta \Big( \frac{1}{2}\int_{\R^d} \nabla(|\nabla \psi_t|^2) \cdot \nabla \Phi \rmd x+\int_{\R^d} \nabla \psi_t \cdot D^2 \Phi \cdot \nabla \psi_t \rmd x  \Big ) \\
  \leq & -\beta \Big( -\frac{1}{2}\int_{\R^d} |\nabla \psi_t|^2 |\Delta \Phi| \rmd x +\int_{\R^d} |\nabla \psi_t|^2 |\Delta \Phi | \rmd x  \Big ) \\
  \leq & -\frac{\beta C_{\Phi}d}{2}  \lambda_t.
  \end{aligned}
  \end{equation*}
  Combining the estimate \eqref{equa5}, the H\"older inequality, and the Young inequality with {\rm (III)}, we deduce that
  \begin{equation*}
  \begin{aligned}
  {\rm (III)}=& -\int_{\R^d} \nabla \psi_t\cdot \nabla \Phi (J_t-\beta)r_t\rmd x \\
  \leq & K_{\var} \|r_t\|_{L^2} \|\nabla \psi_t\|_{L^2} \\
  \leq & \dfrac{\sigma^2}{2}\|r_t\|_{L^2}^2 +\dfrac{K_{\var}^2}{2\sigma^2} \|\nabla \psi_t \|_{L^2}^2 \\
  =& \dfrac{\sigma^2}{2}\|r_t\|_{L^2}^2 +\dfrac{K_{\var}^2}{2\sigma^2} \lambda_t.
  \end{aligned}
  \end{equation*}
  Therefore,
  \begin{equation*}
  \dfrac{1}{2}\dfrac{\rmd}{\rmd t}\lambda_t ={\rm (I)}+{\rm (II)}+{\rm (III)}\leq -\dfrac{\sigma^2}{2}\|r_t\|_{L^2}^2 +C_{\lambda} \lambda_t.
  \end{equation*}
  It follows that
  \begin{equation*}
  \dfrac{\rmd}{ \rmd t}\Big(\rme^{-2C_{\lambda}t}\lambda_t\Big) +\sigma^2 \rme^{-2C_{\lambda}t}\|r_t\|_{L^2}^2\leq 0.
  \end{equation*}
  Integrating over $[0, \tau]$ and using $\lambda_{\tau}\geq 0$, we deduce that
  \begin{equation*}
  -\lambda_0 + \sigma^2 \int_0^{\tau}\rme^{-2C_{\lambda}t}\|r_t\|_{L^2}^2 \rmd t
  \leq \rme^{-2C_{\lambda}\tau}\lambda_{\tau} -\lambda_0 + \sigma^2 \int_0^{\tau}\rme^{-2C_{\lambda}t}\|r_t\|_{L^2}^2 \rmd t \leq 0.
  \end{equation*}
  Consequently,
  \begin{equation*}
  \int_0^{\tau}\|r_t\|_{L^2}^2 \rmd r  \leq \dfrac{\rme^{2C_{\lambda}\tau}}{\sigma^2}\lambda_0,
  \end{equation*}
  which means that
  \begin{equation*}
  \int_0^{\tau}\|p_t-p_{\infty}\|_{L^2}^2 \rmd r  \leq \dfrac{\rme^{2C_{\lambda}\tau}}{\sigma^2}\|p_0-p_{\infty}\|_{\dot{H}^{-1}}^2.
  \end{equation*}

  According to Theorem 5 in \cite{P}, we get
  \begin{equation*}
  \|p_0-p_{\infty}\|_{\dot{H}^{-1}}\leq C_{H,\var} W_2(P_0, P_{\infty}).
  \end{equation*}
  It follows that
  \begin{equation*}
  \int_0^{\tau}\|p_0-p_{\infty}\|_{L^2}^2 \rmd r  \leq C_{LW}W_2^2(P_0, P_{\infty}).
  \end{equation*}
\end{proof}

\begin{thm}[Convergence rate of the modified Wasserstein metric]\label{WHcon}
Suppose that Assumptions \ref{reg}, \ref{met} and \ref{Phicon} hold. Then, for any $P_0\in \Lambda_{\var}^g$, we have
\begin{equation}\label{HWh}
  H_g(P_t|P_{\infty}) \leq \rme^{-\frac{t-\tau}{C_{HI}}}C_{HW} W_h^2(P_0, P_{\infty}),~ T > t\geq \tau.
\end{equation}
Moreover, the convergence rate in the $W_h$ metric is given by
\begin{equation}\label{Whconver}
 W_h^2(P_t|P_{\infty}) \leq \rme^{-\frac{t-\tau}{C_{HI}}} C_T C_{HW} W_h^2(P_0, P_{\infty}),~ T >t\geq \tau,
\end{equation}
where $L_B:=b_2 C_{\Phi}d +\dfrac{L_b^{\var}b_2 s_2}{\sigma^2}$, $C_B:=(L_b^{\var})^2 M_{\var} s_2 C_{LW}$, $C_{HW}:=\dfrac{\gamma_2}{2\sigma^2\int_0^{\tau} \rme^{-2L_B t} \rmd t} +\dfrac{\gamma_2}{2\sigma^2} C_B$.
\end{thm}

\begin{proof}
The idea of the proof is inspired by Section 4 of \cite{W}. We proceed by invoking the Girsanov theorem.

First, we define
$$B^{\infty}(x):=-b(p_{\infty}(x))\nabla \Phi(x), ~B_t^{P}( x):=-b(p_t( x))\nabla \Phi(x)$$
and
$$\bar{B}_t(x):=B^{\infty}(x)-B_t^{P}(x).$$
The estimate $|b(p_t)-b(p_{\infty})|\leq L_b^{\var}|p_t-p_{\infty}|$ implies
\begin{equation*}
|\bar{B}_t(x)|^2\leq (L_b^{\var})^2|\nabla \Phi(x)|^2 |p_t-p_{\infty}|^2.
\end{equation*}
The bound $p_t\leq M_{\var} p_{\infty}$ induces
\begin{equation*}
\int_{\R^d} |\bar{B}_t(x)|^2 p_t\rmd x \leq (L_b^{\var})^2 M_{\var} \int_{\R^d} |\nabla \Phi(x)|^2 |p_t-p_{\infty}|^2 p_{\infty}\rmd x \leq (L_b^{\var})^2 M_{\var} s_2 \|p_t-p_{\infty}\|_{L^2}^2.
\end{equation*}
Using Lemma \ref{lemW}, we derive
\begin{equation}\label{equa6}
\int_0^{\tau} \int_{\R^d} |\bar{B}_t( x)|^2 p_t\rmd x \rmd t \leq C_B W_2^2(P_0, P_{\infty}).
\end{equation}
A direct calculation shows that
\begin{equation*}
\begin{aligned}
DB^{\infty}(x)= &-b(p_{\infty}(x)) D^2 \Phi(x)-b'(p_{\infty}(x))\nabla p_{\infty}(x)\otimes \nabla \Phi(x) \\
= &-b(p_{\infty}(x)) D^2 \Phi(x) + \dfrac{b'(p_{\infty}(x)) h(p_{\infty}(x))}{\sigma^2}\nabla \Phi(x) \otimes \nabla \Phi(x)
\end{aligned}
\end{equation*}
by $0=\nabla g(p_{\infty}) +\nabla \Phi =\frac{\sigma^2}{h(p_{\infty})}\nabla p_{\infty}+\nabla \Phi$.
Then the operator norm estimate
\begin{equation}\label{LB}
\| DB^{\infty}(x) \|_{\rm op}:= \sup_{|v|=1} |DB^{\infty}(x)\cdot v| \leq b_2 C_{\Phi} d +\dfrac{L_b^{\var}b_2}{\sigma^2}|\nabla \Phi(x)|^2 p_{\infty}(x) \leq L_B
\end{equation}
yields
\begin{equation*}
\langle B^{\infty}(x)-B^{\infty}(y), x-y \rangle \leq L_B |x-y|^2.
\end{equation*}

Let $X_0$ and $Y_0$ be two random variables such that
$${\rm Law}(X_0)=P_{\infty},~ {\rm Law}(Y_0)=P_0$$
and
$$\E_{\P}|X_0-Y_0|^2=W_2^2(P_0, P_{\infty}).$$
In addition, let $X_t$ be the solution of
\begin{equation}\label{Xt}
\rmd X_t=B^{\infty}(X_t) \rmd t +\sqrt{2}\sigma \rmd W_t
\end{equation}
with initial condition $X_0$.
Since $p_{\infty}$ is the density of the invariant measure, we have ${\rm Law}(X_t)=P_{\infty}$ for all $t\geq 0$.
Next, we construct the following SDE
\begin{equation}\label{Yt}
\rmd Y_t=B^{P}_t(Y_t) \rmd t +\sqrt{2}\sigma \rmd W_t+\alpha_t \rmd t
\end{equation}
with initial condition $Y_0$, where
\begin{equation*}
\alpha_t:=\vartheta_t e_t+B^{\infty}(Y_t)-B^{P}_t( Y_t)
\end{equation*}
and
\begin{equation*}
\vartheta_t:=\dfrac{|X_0-Y_0|\rme^{-L_Bt}}{\int_0^{\tau}\rme^{-2L_B t}\rmd t}, ~e_t:=
\begin{cases}
 \dfrac{X_t-Y_t}{|X_t-Y_t|}, \quad & X_t-Y_t\ne 0, \\
 0, &  X_t-Y_t=0.
\end{cases}
\end{equation*}
For convenience, denote $Z_t:=X_t-Y_t$.
Then
\begin{equation*}
\rmd Z_t=\big( B^{\infty}(X_t)-B^{\infty}(Y_t) -\vartheta_t e_t \big) \rmd t.
\end{equation*}
Estimate \eqref{LB} ensures
\begin{equation*}
\dfrac{\rmd |Z_t|}{\rmd t} \leq L_B |Z_t|-\vartheta_t.
\end{equation*}
It follows that
\begin{equation*}
\rme^{-L_B\tau}|Z_{\tau}|-|Z_0| \leq -\int_0^{\tau} \rme^{-L_Bt} \vartheta_t \rmd t =-|Z_0|,
\end{equation*}
which means $X_{\tau}=Y_{\tau}$.

Denote
$$\zeta_t:=\dfrac{\alpha_t}{\sqrt{2}\sigma}, ~R_{\tau}:=\rme^{  -\int_0^{\tau} \langle \zeta_t, \rmd W_t \rangle -\frac{1}{2} \int_0^{\tau}|\zeta_t|^2 \rmd t }.$$
The Girsanov theorem implies that, under $\Q:=R_{\tau}\P$,
\begin{equation*}
\widetilde{W}_t:=W_t+\int_0^t \zeta_s \rmd s
\end{equation*}
is a Brownian motion.
Under this probability measure,
\begin{equation*}
  \rmd Y_t=B^{P}_t(Y_t) \rmd t +\sqrt{2}\sigma\rmd \widetilde{W}_t
\end{equation*}
admits a unique solution $Y_t$ and ${\rm Law}_{\Q}(Y_t)=P_t$.
Since $X_{\tau}=Y_{\tau}$, we obtain
\begin{equation*}
(P_{\tau}\log f)(p_0)=\E_{\Q}[\log f(Y_{\tau})]=\E_{\P}[R_{\tau} \log f(Y_{\tau})]=\E_{\P}[R_{\tau} \log f(X_{\tau})]
\end{equation*}
for $f\in \mathcal{B}_b^+$. 
The Young inequality gives
\begin{equation*}
\begin{aligned}
\E_{\P}[R_{\tau} \log f(X_{\tau})] \leq & \log \E_{\P}[f(X_{\tau})] +\E_{\P}[R_{\tau} \log R_{\tau}] \\
= &\log (P_{\tau} f)(p_{\infty})+\E_{\P}[R_{\tau} \log R_{\tau}].
\end{aligned}
\end{equation*}
Thus,
\begin{equation}\label{LH}
(P_{\tau}\log f)(p_0) \leq \log (P_{\tau} f)(p_{\infty})+\E_{\P}[R_{\tau} \log R_{\tau}],
\end{equation}
which is the log-Harnack inequality.
By the definition of $R_{\tau}$ and $|\alpha_t|^2 \leq 2|\vartheta_t|^2 +2|\bar{B}_t(Y_t)|^2$,
we have
\begin{equation*}
\begin{aligned}
\E_{\P}[R_{\tau} \log R_{\tau}]
=& \E_{\P}\big[R_{\tau}\big(-\int_0^{\tau}\langle \zeta_t , \rmd \widetilde{W}_t \rangle+\dfrac{1}{2} \int_0^{\tau} |\zeta_t|^2 \rmd t \big) \big] \\
=&\dfrac{1}{2}\E_{\Q} \int_0^{\tau}|\zeta_t|^2 \rmd t \\
\leq & \dfrac{1}{2\sigma^2}\E_{\Q}\Big[\int_0^{\tau}|\vartheta_t|^2 \rmd t +\int_0^{\tau}|\bar{B}_t(Y_t)|^2 \rmd t \Big] \\
= & \dfrac{1}{2\sigma^2} \E_{\Q}\Big[\dfrac{|Z_0|^2}{\int_0^{\tau} \rme^{-2L_B t} \rmd t}\Big] +\dfrac{1}{2\sigma^2}\int_0^{\tau} \int_{\R^d}|\bar{B}_t(y)|^2 p_t( y)\rmd y \rmd t.
\end{aligned}
\end{equation*}
The equation $\E_{\Q}[|Z_0|^2]=\E_{\P}[R_{\tau}]\cdot \E_{\P}|Z_0|^2=\E_{\P}|Z_0|^2=W_2^2(P_0, P_{\infty})$ and \eqref{equa6} yield
\begin{equation*}
  \begin{aligned}
  \E_{\P}[R_{\tau} \log R_{\tau}]
  \leq 
  C_{HW}' W_2^2(P_0, P_{\infty}),
  \end{aligned}
\end{equation*}
where $C_{HW}':=\frac{1}{2\sigma^2\int_0^{\tau} \rme^{-2L_B t} \rmd t} +\frac{1}{2\sigma^2} C_B$.
Therefore, \eqref{LH} imples
\begin{equation*}
\int_{\R^d} \log f \rmd P_{\tau} \leq \log \int_{\R^d} f \rmd P_{\infty} +C_{HW}' W_2^2(P_0, P_{\infty}).
\end{equation*}

Letting $f=\dfrac{\rmd P_{\tau}}{\rmd P_{\infty}}$, we have
\begin{equation*}
H(P_{\tau}|P_{\infty}) \leq C_{HW}' W_2^2(P_0, P_{\infty}).
\end{equation*}
A direct calculation gives
\begin{equation*}
  H_g(P_{\tau}|P_{\infty}) \leq \dfrac{\gamma_2}{b_1} H(P_{\tau}|P_{\infty})
\end{equation*}
and
\begin{equation*}
  b_1 W_2^2(P_0, P_{\infty}) \leq W_h^2(P_0, P_{\infty}) \leq b_2 W_2^2(P_0, P_{\infty}).
\end{equation*}
It follows that
\begin{equation*}
  H_g(P_{\tau}|P_{\infty}) \leq \gamma_2 C_{HW}' W_h^2(P_0, P_{\infty}) =C_{HW} W_h^2(P_0, P_{\infty}).
\end{equation*}

Combining \eqref{Hconver} and \eqref{Tineq}, we obtain
\begin{equation*}
  H_g(P_t|P_{\infty}) \leq \rme^{-\frac{t-\tau}{C_{HI}}} H_g(P_{\tau}|P_{\infty}) \leq \rme^{-\frac{t-\tau}{C_{HI}}}C_{HW} W_h^2(P_0, P_{\infty})
\end{equation*}
and
\begin{equation*}
  W_h^2(P_t|P_{\infty}) \leq C_T H_g(P_t|P_{\infty}) \leq \rme^{-\frac{t-\tau}{C_{HI}}} C_T C_{HW} W_h^2(P_0, P_{\infty}),
\end{equation*}
which completes the proof.
\end{proof}

\section*{Acknowledgements}
We are grateful to Professor Nicola Gigli for insightful comments.
This work is supported by National Key R\&D Program of China (No. 2023YFA1009200), NSFC (Grants 12531009 and 11925102).


\end{document}